\documentclass{amsart}
\usepackage{amssymb}
\usepackage{graphicx}
\usepackage[import,all]{xypic}

\def\R{\mathbb R}
\def\C{\mathbb C}
\def\T{\mathbb T}
\def\Z{\mathbb Z}

\def\CP{\mathbb C\mathrm P}

\def\n{\mathfrak{n}}

\def\t{\mathfrak{t}}

\def\cone{\mathrm{cone}\,}

\def\Re{\mathrm{Re}\,}
\def\Im{\mathrm{Im}\,}

\def\del{\partial}

\def\zbar{\overline{z}}
\def\wbar{\overline{w}}

\newtheorem{theorem}{Theorem}[section]
\newtheorem{prop}[theorem]{Proposition}
\newtheorem{corollary}[theorem]{Corollary}

\theoremstyle{definition}

\newtheorem{notation}[theorem]{Notation}

\theoremstyle{remark}
\newtheorem{remark}[theorem]{Remark}

\begin{document}
\title{Toric Poisson Structures}
\author{Arlo Caine}

\begin{abstract}
Let $T_\C$ be a complex algebraic torus and let $X(\Sigma)$ be a
smooth projective $T_\C$-variety.  In this paper, a real
$T_\C$-invariant Poisson structure $\Pi_\Sigma$ is constructed on
the complex manifold $X(\Sigma)$, the symplectic leaves of which are
the $T_\C$-orbits in $X(\Sigma)$. It is shown that each leaf admits
a Hamiltonian action by a sub-torus of the compact torus $T\subset
T_\C$. However, the global action of $T_\C$ on
$(X(\Sigma),\Pi_\Sigma)$ is Poisson but not Hamiltonian.  The main
result of the paper is a lower bound for the first Poisson
cohomology of these structures.  For the simplest case,
$X(\Sigma)=\CP^1$, the Poisson cohomology is computed using a
Mayer-Vietoris argument and known results on planar quadratic
Poisson structures and in the example the bound is optimal. The
paper concludes with the example of $\CP^n$, where the modular
vector field with respect to a particular Delzant Liouville form
admits a curious formula in terms of Delzant moment data.  This
formula enables one to compute the zero locus of this modular vector
field and relate it to the Euclidean geometry of the moment simplex.
\end{abstract}

\maketitle


\section{Introduction}

Symplectic structures on compact toric manifolds $X(\Sigma)$  have
been extensively studied (cf. \cite{Delzant}, \cite{Guillemin},
\cite{LermanTolman}) providing a wealth of examples of compact
symplectic manifolds with symmetry. The purpose of this paper is to
construct and study a real Poisson structure $\Pi_\Sigma$ on
$X(\Sigma)$ whose symplectic leaves are the complex torus orbits, so
that $\Pi_\Sigma$ is not regular and yet has a dense open leaf.
Since degenerate Poisson structures play an important role in study
of homogeneous spaces (e.g.: groups \cite{EvensLu2},
\cite{LuWeinstein}; flag varieties \cite{EvensLu2}, \cite{FothLu1},
\cite{WebsterYakimov}; symmetric spaces \cite{CainePickrell},
\cite{FothLu2}, \cite{Xu}), it is interesting to consider such
Poisson structures in the almost homogeneous setting of toric
varieties.  In the natural system of holomorphic coordinates,
$\Pi_\Sigma$ turns out to be a real quadratic Poisson structure.
Structures of these types have found application in combinatorics
(\cite{Gekhtman},\cite{Yakimov}) and have been the focus of first
attempts at developing a theory of local invariants for Poisson
structures (\cite{DufourZung}, \cite{LiuXu}). In addition to the
construction of $\Pi_\Sigma$ and its local form, global properties
are addressed, such as the symplectic geometry of its leaves and the
modular class, the main result being the following lower bound for
the first Poisson cohomology (cf. section \ref{TheModularClass}).

\begin{theorem}
$\dim_\R H^1(X(\Sigma),\Pi_\Sigma)\ge 2n+1$ where $n$ is the complex
dimension of the maximal torus acting effectively on $X(\Sigma)$.
\end{theorem}

A standard example used for illustrating differences between
symplectic and Poisson geometry is the quadratic Poisson structure
in the plane
\[\Pi=\textstyle{\frac{1}{2}}(x^2+y^2)\del_x\wedge \del_y.\]
It is degenerate only at the origin and the open symplectic leaf has
interesting topology.  Motivated by questions in \cite{Vaisman},
Nakanishi in \cite{Nakanishi} computed the Poisson cohomology of
this structure showing that it was finite dimensional and that, in
particular, $\dim_\R H^1(\R^2,\Pi)=2$.  It is invariant under the
action of the circle $\T$ by rotations of the plane about the origin
and on the open symplectic leaf $\R^2\setminus\{0\}$, this action is
Hamiltonian. However, the global action of $\T$ on $(\R^2,\Pi)$ is
Poisson but not Hamiltonian as the momentum map on
$\R^2\setminus\{0\}$ fails to extend to the origin, having a
logarithmic singularity there. This generates one dimension in
$H^1(\R^2,\Pi)$.  The other turns out to be generated by the radial
vector field $x\del_x+y\del_y$ on the plane.

The source of these infinitesimal outer automorphisms can be seen by
rewriting $\Pi$ in complex coordinates $z=x+iy$. Then $\Pi=i|z|^2
\del_{\zbar}\wedge \del_{z}$ and a short calculation shows that
$\Pi$ is invariant under the action of the complex torus $\T_\C$.
Indeed, if $\zeta\in \T_\C$, then $z\mapsto z\zeta=z'$ sends
$z\del_z\mapsto z'\del_{z'}$ and thus preserves $\Pi$.  The local
Hamiltonian functions for the action of $\T_\C/\T$ fail to extend to
the entirety of the open symplectic leaf for topological reasons,
let alone to all of $\C$. So, the real Lie algebra $\R+i\R$ of the
complex torus $\T_\C$ is included in $H^1(\C,\Pi)$.

In this paper, this example is generalized to Poisson structures on
smooth compact toric varieties.  Let $T$ be a compact abelian Lie
group.  Let $\Lambda^*\subset\t^*$ denote the dual group of $T$
viewed as lattice in the dual of its Lie algebra $\t$.  Given a
Delzant polytope $\Delta$ in $\t^*$, one can construct a smooth
algebraic manifold $X(\Sigma)$ from the data of its dual fan
$\Sigma$. The complexified torus $T_\C$ acts on $X(\Sigma)$ with an
open dense orbit. A version of this construction is presented in
section \ref{Construction_of_Varieties} in order to establish some
notation to be used in later sections.  The exposition there is a
blend of similar discussions in \cite{Proudfoot}, and chapter VII of
\cite{Audin}, in the smooth projective case.  Although standard, it
is included here for completeness as results in subsequent sections
depend on the details of this construction.

Roughly, $X(\Sigma)$ arises as the quotient of $\C^d$, where $d$ is
the number of facets of $\Delta$, by the action of a complex
algebraic torus.  The Poisson structure $\Pi_\Sigma$ is that
co-induced by the quotient map from the Poisson structure
$\Pi\oplus\Pi\oplus \dots \oplus\Pi$ on $\C^d$. On $X(\Sigma)$ there
is a distinguished system of complex algebraic coordinate charts
associated to the vertices of $\Delta$ and in these coordinates
$\Pi_\Sigma$ is a homogeneous quadratic Poisson structure.  The
symplectic leaves of the prototype on $\C^d$ are precisely the
orbits of $\T_\C^d$ on $\C^d$. Perhaps not surprising then is the
result that the symplectic leaves of $\Pi_\Sigma$ on $X(\Sigma)$
turn out to be the $T_\C$-orbits in $X(\Sigma)$.  By construction,
the action of $T_\C$ on $(X(\Sigma),\Pi_\Sigma)$ is Poisson, but no
subgroup of $T_\C$ acts in a globally Hamiltonian way.  However,
each symplectic leaf admits a Hamiltonian action by a sub-torus of
the compact torus $T$. This is the content of section
\ref{ToricPoissonStructures}.

Considering the results of Nakanishi in the planar case, a natural
question to ask is whether the image of $\t+i\t$ in
$H^1(X(\Sigma),\Pi_\Sigma)$ generates it. In section
\ref{TheModularClass}, it is shown that this is not the case.  The
modular class of $(X(\Sigma),\Pi_\Sigma)$ is always non-trivial and
and independent of the image of $\t+i\t$. While this does not pin
down $H^1(X(\Sigma),\Pi_\Sigma)$ in general, it does exhaust the
classes for the simplest toric variety, $\CP^1$. In section
\ref{Poisson_Cohomology_CP1}, the Poisson cohomology of
$(\CP^1,\Pi_\Sigma)$ is computed using a Mayer-Vietoris argument and
the theorems of Nakanishi.

Of course, much has been written about the connection between smooth
projective $T_\C$-varities and smooth compact connected symplectic
manifolds of dimension $2\dim  T$ admitting a Hamiltonian action by
$T$. Through the fundamental work of Delzant \cite{Delzant} such
symplectic manifolds are characterized by their momentum polytope
$\Delta\subset\t^*$ in that each is isomorphic to $X(\Sigma)$ (where
$\Sigma$ is the dual fan of $\Delta$) equipped with K\"ahler metric
with respect to which the action of $T$ is Hamiltonian with momentum
map image $\Delta$.  The family of Delzant polytopes having a given
dual fan thus produces a family of K\"ahler metrics on $X(\Sigma)$.
It has thus far proved difficult to assess the compatibility of
$\Pi_\Sigma$ with a Poisson structure $\pi_\Delta$ determined by
Delzant's construction. However, the final section of this paper
contains a curious formula for the modular vector field of
$\Pi_\Sigma$ on $\CP^n$ with respect to a particular Delzant
Liouville volume $\frac{1}{n!}\omega_\Delta^n$ in terms of Delzant
momentum map data.  A consequence of this formula is a
characterization of the zero set of the modular vector field in
terms of the centroids of the faces of the momentum simplex.

Before continuing to the next section, some notation and conventions
to be used in the paper are fixed.  Throughout the paper $\T$
denotes the group of complex numbers of modulus one and $h$ will
denote a fixed, but arbitrary, positive real parameter. The Lie
algebra of $\T$ will be identified with $\R$ by the map $\R\to \T:
s\mapsto \exp(\frac{i}{\hbar}s)$ where $\hbar$ is $h$ divided by the
circumference of the unit circle.\footnote{In dealing with analysis
on the unit circle, factors of its circumference show up everywhere
in computations and different authors have used different
conventions in attempts to deposit these factors in one place or
another. The reason for inserting the parameter $h$ and the symbol
$\hbar$ here is simply to keep track of those factors while freeing
up the symbol $\pi$ to represent other things.} The kernel of this
map is the lattice $h\Z\subset \R$. By taking products, this induces
an identification of the Lie algebra of $\T^d$ with $\R^d$ and the
kernel of the product map $\R^d\to \T^d$ is the lattice
$(h\Z)^d\subset \R^d$. For a convex set $F$ in a real vector space,
$F^\circ$ will denote the relative interior, i.e., its interior as a
topological subspace of its affine hull.

If $\omega$ is a 2-form on a smooth manifold $M$, we denote by
$\omega^\#\colon TM\to T^*M$ the bundle map defined by contraction
in the \emph{second argument}, $v\mapsto
\omega^\#(v):=\omega(\cdot,v)$. If $\pi$ is a bi-vector field on $M$
then $\pi^\#\colon T^*M\to T^{**}M$ is defined by contraction in the
first argument, $\nu\mapsto \pi^\#(\nu):=\pi(\nu,\cdot)$.  If $M$ is
finite dimensional, as are all the spaces considered in this
article, then there is a canonical bundle isomorphism $T^{**}M\simeq
TM$.  In that case, if $\omega$ is a symplectic form then
$\omega^\#$ is an isomorphism at each point and the image of
$\omega$ under the map $\bigwedge^2 T^*M\to \bigwedge^2T^{**}M$
induced by $(\omega^\#)^{-1}$ can be canonically identified with a
non-degenerate Poisson structure $\pi$ on $M$. With these
conventions, $\pi^\#\circ \omega^\#\colon TM\to TM$ is the identity
and the Hamiltonian function for the standard Poisson structure
generating rigid counter-clockwise rotation of the plane is
positive, i.e., if $\pi=\del_x\wedge \del_y$, then
$\pi^\#(d(\frac{1}{2}(x^2+y^2)))=-y\del_x+x\del_y$.

We will regard $\T_\C$ as acting on $\C$ by right multiplication,
$\C\times \T_\C\to \C: (z,\zeta)\mapsto z\zeta$. As a real vector
space, $\C$ will be identified with $\R^2$ via the map $z\mapsto
\Re(z)e_1+\Im(z)e_2$.  The conventions adopted here for the relation
between tensors in this complex coordinate and in the real
coordinates $x=\Re(z)$ and $y=\Im(z)$ are the following. Set
$dz=dx+idy$ and $d\zbar = dx-idy$, and $ \del_{z}
=\textstyle{\frac{1}{2}}(\del_{x}-i\del_{y})$ and $\del_{\zbar} =
\textstyle{\frac{1}{2}}(\del_{x}+i\del_{y})$ so that
$\del_{x}=\del_{z}+\del_{\zbar}=2\Re(\del_{z})$ and
$\del_{y}=(i\del_{z}-i\del_{\zbar})=2\Re(i\del_{z})$ and
$\del_z,\del_{\zbar}$ is dual as a complex basis to $dz,d\zbar$. In
short, the holomorphic cotangent bundle has been identified with the
real cotangent bundle using the real projection, whereas the
holomorphic tangent bundle has been identified with the real tangent
bundle using twice the real projection. Under the identifications,
the field $z\del_z+c.c.$ is radial vector field $x\del_x+y\del_y$
and $iz\del_z+c.c.$ is the vector field generating rigid
counter-clockwise rotation about the origin.

The organization of the paper is as follows.  Section
\ref{Construction_of_Varieties} reviews the construction of smooth
compact $T_\C$-manifolds.  The Delzant construction is reviewed in
section \ref{Delzant_Construction} for later use in section
\ref{ModVectorFieldonCPn} where a formula involving Delzant moment
data is derived for the modular vector field of $\Pi_\Sigma$ on
$\CP^n$ relative to the Delzant Liouville volume form.  The Poisson
structure $\Pi_\Sigma$ is constructed in section
\ref{ToricPoissonStructures} and its local geometry and the
symplectic geometry of its leaves are analyzed.  The main result on
the first Poisson cohomology of $(X(\Sigma),\Pi_\Sigma)$ is proved
in section \ref{TheModularClass}, and in section
\ref{Poisson_Cohomology_CP1} the Poisson cohomology of
$(\CP^1,\Pi_\Sigma)$ is computed by a Mayer-Vietoris argument.

The author would like to thank Sam Evens, Mark Colarusso, Misha
Gekhtman, and Stephan Stolz for useful discussions.


\section{Compact smooth algebraic
$T_\C$-manifolds.\label{Construction_of_Varieties}}

Let $T$ be a compact abelian Lie group of dimension $n$.  Write
$\Lambda^*$ for the dual group of $T$ viewed as a lattice in $\t^*$,
the dual of the Lie algebra $\t$ of $T$, and let $\Lambda$ denote
the dual lattice in $\t$.  Let $\t_\C$ and $T_\C$ denote the
complexifications of $\t$ of $T$, respectively.

Let $\Delta$ be a convex polytope in $\t^*$.  An edge emanating from
a vertex $\nu\in \Delta$ is said to be generated by $\eta\in \t^*$
if and only if the edge is contained in the ray $\{\nu+t\eta:
t>0\}$.  A convex polytope $\Delta$ in $\t^*$ is \emph{Delzant} if
and only if for each vertex of $\nu$ of $\Delta$, the edges
emanating from $\nu$ are generated by a basis of $\Lambda^*$.  Given
a Delzant polytope $\Delta$ in $\t^*$, let $\Sigma$ denote the fan
in $\Lambda$ over its dual polytope (cf. \cite{Fulton} \S 1.5).  To
$\Sigma$ one can associate a smooth complex projective variety
$X(\Sigma)$ on which $T_\C$ acts with an open orbit.  As a complex
manifold, $X(\Sigma)$ may be constructed as the quotient an open
subset of $\C^d$ by the free action of a complex torus.

The combinatorial properties of the polytope $\Delta$, and hence of
the dual fan $\Sigma$, are the essential data used in the
construction.  So, there are two dual points of view one can take,
using either then polytope or the fan to index the relevant
quantities.  From the point of view of algebraic geometry, it is
natural to use fan as the constructions can be generalized to
singular toric varieties whose fans are not the dual fan of a
Delzant polytope.  From the point of view of symplectic geometry,
however, the polytope is natural choice.  The subject of Poisson
geometry overlaps both of these worlds, so this article will adopt
the parallel use of both the polytope and its dual fan.

Let $\Delta$ be a Delzant polytope in $\t^*$ and enumerate the
facets $1,2,\ldots, d$ of $\Delta$. Correspondingly, this enumerates
the elements of the 1-skeleton of $\Sigma$, denoted $\Sigma^{(1)}$.
This enumeration sets up an inclusion reversing correspondence
between the set of faces of $\Delta$ and a subset of the power set
of $\{1,2,\ldots,d\}$, partially ordered by inclusion. Namely, each
face $F$ of codimension $f$ corresponds to the subset
$\{s_1,\ldots,s_f\}$ of labels of the facets whose intersection is
$F$. For the $\ell^{th}$ ray in $\Sigma^{(1)}$, let $u_\ell$ be the
unique primitive element of $\Lambda$ which generates it.  Then the
cone in $\Lambda$ generated by $\{u_{s_1},u_{s_2},\ldots,u_{s_f}\}$,
denoted $\cone\{u_{s_1},\ldots,u_{s_f}\}$, is the strongly convex
rational polyhedral cone in $\Sigma$ determined by $F$.  As $\Delta$
necessarily has dimension $n$, each vertex $V$ corresponds to a
subset $\{s_1,\ldots,s_n\}$ of cardinality $n=\dim \t$.  Order
$u_{s_1},\ldots,u_{s_n}$ by requiring that $s_1<s_2<\ldots <s_n$,
then $u_{s_1},\ldots,u_{s_n}$ is a basis of $\t$ and the dual basis
of $\t^*$ generates the edges of $\Delta$ emanating from the vertex
$V$.

\begin{notation}\label{notation_structure_map_of_polytope}
Define a $\Z$-linear map $p\colon (h\Z)^d\to \Lambda$ by the
assignments $he_\ell\mapsto u_\ell$ for each $\ell =1,2,\ldots,d$
and also write $p\colon \R^d\to \t$ and $p\colon \C^d\to \t_\C$ for
the extensions to real and complex scalars, respectively.  Up to the
enumeration of the facets of $\Delta$ (or equivalently the rays in
$\Sigma^{(1)}$) the map $p$ is uniquely associated with $\Delta$
(and thus $\Sigma$).  Let $\n\subset \R^d$ denote the kernel of
$p\colon \R^d\to \t$ so that $\n\cap (h\Z)^d$ is the kernel of
$p\colon (h\Z)^d\to \Lambda$ and $\n_\C=\ker(p\colon \C^d\to \t_\C)$
is the complexification of $\n$.  Let $\iota\colon \n\cap (h\Z)^d\to
(h\Z)^d$ denote the inclusion and, as before, also write
$\iota\colon \n\to \R^d$ and $\iota\colon \n_\C\to \C^d$ for the
extensions to real and complex scalars, respectively.
\end{notation}

Let $N_\C\subset \T_\C^d$ denote the complex algebraic subgroup of
$\T_\C^d$ corresponding to $\n_\C$. Essentially, $X(\Sigma)$ is the
quotient of $\C^d$ by the action of $N_\C$.  More precisely, it is
the quotient by $N_\C$ of a dense open subset $\mathcal
U_\Sigma\subset \C^d$ on which $N_\C$ acts freely. As $\mathcal
U_\Sigma$ is defined as a union of $\T_\C^d$ orbits on $\C^d$,
notation for describing these must be introduced.

\begin{notation}\label{notation_torus_orbits}
Given $S\subset \{1,2,\ldots, d\}$, let
\[(\C_0)^S=\{z=(z_1,\ldots,z_d)\in \C^d\colon z_\ell
=0\Leftrightarrow \ell\in S^c\}\] where $S^c$ denotes the complement
of $S$.
\end{notation}
For $d=1$, the orbits of $\T_\C$ are the origin $(\C_0)^\emptyset$
and its complement $(\C_0)^{\{1\}}$. By taking products it is clear
that the assignment $S\mapsto (\C_0)^S$ gives a bijection between
the orbits of $\T_\C^d$ on $\C^d$ and the powerset of $\{1,2,\ldots,
d\}$.  It is important to note that the closure of $(\C_0)^S$ is the
union of the orbits $(\C_0)^{S'}$ indexed by $S'\subset S$, and thus
equals the linear subspace $\C^S\subset \C^d$ spanned by
$\{e_s\colon s\in S\}$.  The isotropy subgroup of $(\C_0)^S$ is
$\T_\C^{S^c}:=\exp(\C^{S^c})$.

\begin{notation}\label{DefofU_Sigma2}
Let
\begin{equation}\label{toric_locus2}
\mathcal U_\Sigma := \bigcup_S \,(\C_0)^S \subset \C^d
\end{equation}
where the union is taken over $S\subset \{1,2,\ldots,d\}$ such that
$S^c$ labels a face of $\Delta$ (equivalently, the subsets $S$ such
that $\cone \{u_s\colon s\in S^c\}\in \Sigma$).
\end{notation}

\begin{prop}\label{NactionFree2}
$\mathcal U_\Sigma$ is an open subset of $\C^d$ on which $N_\C$ acts
freely.
\end{prop}

\begin{proof}
The union in (\ref{toric_locus2}) is disjoint since $(\C_0)^{S_1}$
and $(\C_0)^{S_2}$ are different $\T_\C^d$-orbits if $S_1\not= S_2$.
In fact, as $\C^d$ is a disjoint union of $\T_\C^d$-orbits,
\begin{equation}\label{complement_of_closed2}
\mathcal U_\Sigma = \C^d\setminus \bigcup_{S'} (\C_0)^{S'}
\end{equation}
where the union is over the subsets $S'$ such that $(S')^c$
\emph{does not} label a face of $\Delta$.  Note that if $S\subset
S'$ and $(S')^c$ does not label a face of $\Delta$ then neither does
$S$ because $S^c\supset (S')^c$.  Hence, if $(S')^c$ does not label
a face of $\Delta$, then the closure of $(\C_0)^{S'}$, i.e.,
\[
\C^{S'}:=\bigcup_{S\subset S'} \C_0^S,
\]
is contained in the union $\bigcup_S (\C_0)^S$ over sets $S$ for
which $S^c$ does not label a face of $\Delta$.  It follows that the
set being deleted from $\C^d$ in (\ref{complement_of_closed2}) is a
closed set, and hence $\mathcal U_\Sigma$ is open.

To prove that the $N_\C$-action is free, one needs to show that the
isotropy subgroup of $N_\C$ at each point of $\mathcal U_\Sigma$ is
trivial. Consider one $\T_\C^d$-orbit $(\C_0)^S$ in $U_\Sigma$.  The
isotropy subgroup of each $z\in (\C_0)^S$ is $N_\C\cap \T_\C^{S^c}$.
If $S^c$ labels a face of $\Delta$ then the set $\{u_s\colon s\in
S^c\}$ is a subset of an integral basis of $\Lambda$.  Thus
$N_\C\cap \T_\C^{S^c}=1$.
\end{proof}

\begin{notation}\label{notation_toric_variety}
The quotient space $\mathcal U_\Sigma/N_\C$ is a smooth complex
manifold with a right action of the complex torus $T_\C\simeq
N_\C\backslash \T_\C^d$.  Denote this quotient by $X(\Sigma)$.
\end{notation}

The following proposition is obvious.

\begin{prop}
The $T_\C$-orbits in $X(\Sigma)$ of complex codimension $\ell$ are
in bijection with the faces of the polytope $\Delta$ of real
codimension $\ell$ (equivalently, the cones in $\Sigma$ of dimension
$\ell$) for each $\ell=0,1,...,n$. In particular, the interior of
$\Delta$ (the zero cone in $\Sigma$) corresponds to an open dense
orbit of $T_\C$ and the vertices of $\Delta$ (the $n$-dimensional
cones in $\Sigma$) correspond to the fixed points of $T_\C$.
\end{prop}

The remainder of this section is devoted to the construction of a
distinguished atlas of local holomorphic coordinates for
$X(\Sigma)$. Given $z=(z_1,\ldots,z_d)\in \mathcal U_\Sigma\subset
\C^d$, write $[z_1\colon z_2\colon \ldots \colon z_d]$ for the class
$zN_\C$ in $X(\Sigma)$.

\begin{prop}
Given a vertex $V$ of $\Delta$ (or an $n$-dimensional cone of
$\Sigma$), let $\mathcal U_V=\cup_S (\C_0)^S$ where the union is
taken over all subsets $S\subset \{1,2,\ldots, d\}$ such that
$S^c\subset S_V$, i.e., $S^c$ labels a face containing $V$ ($S^c$
labels a cone in $\Sigma$ contained in the cone corresponding to
$V$).  Then $\{\mathcal U_V/N_\C\colon V\text{ is a vertex of
}\Delta\}$ is an open cover of $X(\Sigma)$.
\end{prop}
\begin{proof}
Let $V$ be a vertex of $\Delta$.  The faces of $\Delta$ containing
$V$ are in bijection with the subsets $S^c\subset S_V$.  Hence
$\mathcal U_V\subset \mathcal U_\Sigma$.  The complement is the
union
\begin{equation}\label{VertexChart_complement}
\mathcal U_\Sigma\setminus \mathcal U_V=\bigcup_S (\C_0)^S
\end{equation}
where the union is taken over the subsets $S$ such that $S^c$ labels
a face of $\Delta$, but $S^c\not\subset S_V$.  If $S^c\not\subset
S_V$, then every set containing $S^c$ does not contain $S_V$.  Thus,
the closure of $(\C_0)^S$ belongs to the union in
(\ref{VertexChart_complement}).  Being equal to a finite union of
closed sets, $\mathcal U_\Sigma\setminus \mathcal U_V$ is therefore
closed.  Hence $\mathcal U_V$ is open in $\mathcal U_\Sigma$ and
$\mathcal U_V/N_\C$ is open in $X(\Sigma)$.

For each vertex, there is at least one facet which does not meet it.
Hence, the sets $\mathcal U_V$ such that $V$ is a vertex of $\Delta$
cover $\mathcal U_\Sigma$. The result follows.
\end{proof}

\begin{prop}
For each vertex $V$ of $\Delta$ (or $n$-dimensional cone of
$\Sigma$), the principal $N_\C$-bundle $\mathcal U_\Sigma\to
X(\Sigma)$ is trivial over the open set $\mathcal U_V/N_\C$.
\end{prop}
\begin{proof}
Let $V$ be a vertex of $\Delta$.  No generality is lost in assuming
that $S_V=\{1,2,\ldots, n\}$ as the facets of $\Delta$ may be
relabeled if necessary.  Then $\mathcal U_V=\{z\in \C^d\colon
z_{\ell}\not=0, \ell=n+1,n+2,\ldots,d\}$.

As $u_1,u_2,\ldots, u_n$ form a basis for $\Lambda$, there exists a
unique integral $n\times (d-n)$ matrix $A$ such that
$u_{n+\ell}+\sum_{k=1}^n A_{k\ell}u_{k\ell}=0$ for each
$\ell=1,2,\ldots,d-n$.  The vectors $hc_1,\ldots,hc_{d-n}$ where
$c_\ell=e_{n+\ell}+\sum_{k=1}^n A_{k\ell}u_\ell$ then form a basis
for the kernel of $p\colon (h\Z)^d\to \Lambda$.  The induced
isomorphism $\T_\C^{d-n}\simeq N_\C$ then parameterizes the action
of $N_\C$ on $\C^d$ so that
$\zeta=(\zeta_1,\zeta_2,\ldots,\zeta_{d-n})\in \T_\C^{d-n}$ acts on
$z=(z_1,\ldots,z_d)\in \C^d$ by
\[
z.\zeta =
(z_1\zeta^{a_1},\ldots,z_n\zeta^{a_n},z_{n+1}\zeta_1,\ldots,z_d\zeta_{d-n})
\]
where $\zeta^{a_\ell}$ denotes the monomial $\zeta_1^{a_{\ell
1}}\dots \zeta_{d-n}^{a_{\ell (d-n)}}$.  Thus, the quotient map
$\mathcal U_V\to \mathcal U_V/N_\C$ has the form
\[
(z_1,\ldots, z_d)\mapsto
[z_1\tilde{z}^{-a_1}:\ldots:z_n\tilde{z}^{-a_n}:1:\ldots:1].
\]
where $\tilde{z}^{-a_\ell}$ denotes the monomial $z_{n+1}^{-a_{\ell
1}}\dots z_{n+d}^{-a_{\ell (d-n)}}$.  The map $\mathcal U_V/N_\C\to
\mathcal U_\Sigma$ defined by $[w_1:\ldots:w_n:1:\dots:1]\to
(w_1,\ldots, w_n,1,\ldots,1)$ clearly gives a section, trivializing
the $N_\C$-bundle.
\end{proof}

\begin{notation}\label{VertexCoordinatesNotation}
The previous two propositions show that there is a distinguished
coordinate chart $w_V\colon \mathcal U_V/N_\C\to \C^n$ associated to
every vertex of $\Delta$ (or $n$-dimensional cone in $\Sigma$).
Given a vertex $V$ labeled by $S_V=\{s_1,\ldots,s_n\}$,
$s_1<s_2<\ldots<s_d$ let $\sigma$ denote the permutation sending
$k\mapsto s_k$ for $k=1,2,\ldots,n$ and $n+\ell\mapsto s_{n_\ell}$
for $\ell=1,2,\ldots,d-n$ where $s_{n+1}<\ldots<s_d$. Find the
unique set of integers $a_{k\ell}$ such that
$u_{s_{n+\ell}}+\sum_{k=1}^n a_{k\ell}u_{s_k}=0$ for each
$\ell=1,2,\ldots,d-n$. Then define $w_V\colon \mathcal U_V/N_\C\to
\C^n$ by $zN_\C\mapsto (w_1,\ldots,w_n)$ where
$w_\ell=z_{s_\ell}\tilde{z}^{-a_\ell}$ and $\tilde{z}^{-a_\ell}$ is
the monomial $z_{s_{n+1}}^{-a_{\ell 1}}\dots z_{s_d}^{-a_{\ell
(d-n)}}$.  In shorthand notation, $w=z^{[1|-A]\sigma}$.
\end{notation}

 It should be noted that in the coordinates associated
to vertex $V$ labeled by $S_V=\{s_1,\ldots,s_n\}$, the elements
$u_{s_1},\ldots,u_{s_n}\in\t$ act by the vector fields
\[(h/\hbar)iw_1\del_{w_1}+c.c., \ldots, (h/\hbar)iw_n\del_{w_n}+c.c.,\] respectively.
Thus, the open $T_\C$ orbit is characterized by the algebraic
inequality $w_1\dots w_n\not=0$.


\section{The Delzant Construction\label{Delzant_Construction}}

The original symplectic convexity theorem of Atiyah and
Guillemin-Sternberg asserts that the image of the momentum map
$\Phi\colon X\to \t^*$ of a Hamiltonian torus action by $T$ on a
compact connected symplectic manifold $(X,\omega)$ is a convex
polytope $\Delta\subset \t^*$.  In \cite{Delzant}, Delzant
determined the geometry of $\Delta$ under the additional assumptions
that the $T$-action is effective and $\dim X=2\dim T$. What is more,
given $\Delta$ he produced a construction of a compact connected
Hamiltonian $T$-manifold (which is also K\"ahler)
$(X_\Delta,\omega_\Delta,\Phi_\Delta)$ of dimension $2\dim T$ on
which $T$ acted effectively with $\Phi_\Delta(X_\Delta)=\Delta$.
Further, he showed that any other such Hamiltonian $T$-manifold
$(X,\omega,\Phi)$ with $\Phi(X)=\Delta$ was necessarily
$T$-equivariantly symplectomorphic to
$(X_\Delta,\omega_\Delta,\Phi_\Delta)$.  A consequence of a theorem
of Frances Kirwan (cf. Theorem 7.4 in \cite{FKirwan}) is that
$X_\Delta$ is homeomorphic to the compact smooth $T_\C$-manifold
$X(\Sigma)$ constructed from the dual fan $\Sigma$ of $\Delta$.  For
torus actions, every range translate of a momentum map is again a
momentum map.  Thus, Delzant's construction yields a family of
K\"ahler metrics on $X(\Sigma)$ parameterized by the translation
classes of Delzant polytopes in $\t^*$.

In the briefest possible terms, here is the construction.  Let
$\Delta\subset \t^*$ be a Delzant polytope with $d$ facets.  When
$\C^d$ is equipped with the non-degenerate Poisson structure
\begin{equation}
\pi=\sum_{\ell=1}^d i\del_{\zbar_\ell}\wedge \del_{z_\ell},
\end{equation}
the standard action of the real torus $\T^d$ on $\C^d$ is
Hamiltonian and $J\colon \C^d\to (\R^d)^*$ given by
\begin{equation}
J(z)=\sum_{\ell=1}^d
\textstyle{\frac{1}{\hbar}}|z_\ell|^2\epsilon_\ell
\end{equation}
is a momentum map. Let $\omega$ denote the symplectic form induced
by $\pi$.  Fixing an enumeration of the facets of $\Delta$, define
$p$, $\iota$, and $\n$, as in Notation
\ref{notation_structure_map_of_polytope}. Taking duals, the sequence
\begin{equation}
\xymatrix{ 0 & \n^* \ar[l] & (\R^d)^*\ar[l]_{\iota^*} & \t^*
\ar[l]_{\phantom{aaa}p^*} & 0 \ar[l]}
\end{equation}
is exact.  The composition $\mu=\iota^*\circ J\colon (\R^d)^*\to
\n^*$ is a momentum map for the action of $N=\exp(\n)\subset \T^d$
on $(\C^d,\pi)$.

For $\lambda \in (\R^d)^*$, write $\lambda\ge 0$ if $\lambda$ takes
non-negative values on the cone spanned by $\{e_1,\ldots,e_d\}$ in
$\R^d$ and write $P^*=\{\lambda\in (\R^d)^*\colon \lambda\ge 0\}$.
Assume that $0\in \Delta^\circ$, then there exists a unique
$\lambda\in (P^*)^\circ$ such that
\begin{equation}\label{Notation_rep_of_Delta}
\Delta=\{\nu\in \t^*\colon hp^*(\nu)+\lambda\ge 0\}.
\end{equation}
The method of Delzant is to construct $(X_\Delta,\omega_\Delta)$ by
symplectic reduction of the Hamiltonian $N$-manifold
$(\C^d,\omega,\mu)$ at the $\mu$-regular value
$\kappa=\iota^*(\lambda)$.

The image under $J$ of the level set $\mu^{-1}(\kappa)$ is precisely
the image of $\Delta\subset \t^*$ under the affine map $\nu\mapsto
hp^*(\nu)+\lambda$.  As $\mu$ is $\T^d$-invariant, there exists a
smooth $T\simeq N\backslash \T^d$-invariant map $\Phi_\Delta\colon
\mu^{-1}(\kappa)/N\to \t^*$ making the diagram
\begin{equation}
\xymatrix{ \mu^{-1}(\kappa) \ar[d] \ar[r]^J & (\R^d)^* \\
\mu^{-1}(\kappa)/N \ar[r]^{\phantom{aaa}\Phi_{\Delta}} & \t^*
\ar[u]_{hp^*+\lambda}}
\end{equation}
commute.  By symplectic reduction, the quotient
$X_\Delta=\mu^{-1}(\kappa)/N$ obtains a symplectic structure from
$\omega$.  Multiplying the induced Poisson structure by the
parameter $h$, one obtains a non-degenerate Poisson structure
$\pi_\Delta$ on $X_\Delta$ with respect to which the $T$-action is
Hamiltonian with momentum map $\Phi_\Delta$. By construction,
$\Phi_\Delta(X_\Delta)=\Delta$.

The level set $\mu^{-1}(\kappa)$ is the boundary of a convex set of
dimension greater than 1, and thus is connected.  Moreover, the map
$J$ is proper and carries $\mu^{-1}(\kappa)$ onto the convex
polytope $hp^*(\Delta)+\lambda\subset (\R^d)^*$ and is thus compact.
Thus $X_\Delta$ is compact and connected and hence so is $X(\Sigma)$
by Kirwan's theorem.

Of importance for this paper is the fact that $\Phi_\Delta$ provides
a correspondence between the $T_\C$-orbits in $X(\Sigma)$ and the
faces of $\Delta$.  In particular, every $T_\C$-orbit is of the form
$\Phi_\Delta^{-1}(F^\circ)$ for a unique face $F$ of $\Delta$, and
$T_\C$ orbits of dimension $\ell$ correspond to faces of dimension
$\ell$.

\begin{remark}
The assumption that $0\in \Delta^\circ$ can be dropped if one first
translates $\Delta$ to $\Delta'=\Delta-\nu$ for some $\nu\in
\Delta^\circ$, so that $0\in (\Delta')^\circ$, then applies the
construction above to get $\omega_{\Delta'}$ and $\Phi_{\Delta'}$
and sets $\omega_{\Delta}=\omega_{\Delta'}$ and
$\Phi_{\Delta}=\Phi_{\Delta'}+\nu$.
\end{remark}


\section{Toric Poisson Structures\label{ToricPoissonStructures}}

Let $\Pi$ denote bi-vector field
\begin{equation}\label{def_of_Pi}
\Pi=\sum_{\ell=1}^d i|z_\ell|^2\del_{\zbar_\ell}\wedge \del_{z_\ell}
\end{equation}
on $\C^d$.  In the case that $d=1$, $\Pi$ defines a Poisson
structure in the plane.

\begin{theorem}
The bi-vector field $\Pi$ defines a Poisson structure on $\C^d$
which is invariant under the action of the complex torus $\T_\C^d$.
\end{theorem}
\begin{proof}
For $d=1$, $\Pi=i|z|^2\del_{\zbar}\wedge \del_z$ is automatically
Poisson being a bi-vector field on a space of real dimension $2$.
The bi-vector field in (\ref{def_of_Pi}) is the product Poisson
structure induced on $\C^d$ and is therefore Poisson. Alternatively,
one could note that the map $\exp\colon \C^d\to \C^d$, defined by
$(z_1',\ldots,z_d')\mapsto (e^{z_1'},\ldots,e^{z_d'})$ has open
dense image and carries the Poisson structure $\pi$ to $\Pi$.  As
$\Pi$ is smooth $[\Pi,\Pi]=\exp_*[\pi,\pi]=0$ on an open dense set,
$\Pi$ is Poisson.  With this point of view, the invariance of $\Pi$
under the action of $\T_\C^d$ is an immediate consequence of the
translation invariance of $\pi$ on $\C^d$.
\end{proof}

It is interesting that this quadratic Poisson structure arises
naturally as the image of a constant Poisson structure under the
exponential map.  In dimension $d=1$, $\Pi=i|z|^2\del_{\zbar}\wedge
\del_z$ has precisely two symplectic leaves, the origin in $\C$ and
its complement $\C_0$. Given the arguments in the previous proof,
the following theorem is immediate.

\begin{theorem}
The symplectic leaves of $(\C^d,\Pi)$ are precisely the orbits of
$\T_\C^d$ on $\C^d$.
\end{theorem}

In particular, although this Poisson structure is not regular, the
symplectic leaves are finite in number and can be enumerated by the
$2^d$ subsets of the set of indices $\{1,2,\ldots,d\}$, $S\mapsto
(\C_0)^S$ as in the previous section.

\begin{theorem}
The action of $\T_\C^d$ on $(\C^d,\Pi)$ is Poisson but not
Hamiltonian.  However, each symplectic leaf $(\C_0)^{S}$ admits a
Hamiltonian action by the real sub-torus $\T^{S}$ of $\T^d$.
\end{theorem}
\begin{proof}
Due to the product structure of $\Pi$ it suffices to establish the
first claim in dimension $d=1$.  The infinitesimal action of $\T$ on
$\C$ is generated by the vector field $\frac{i}{\hbar}z\del_z+c.c.$
The function $\frac{1}{\hbar}\log |z|^2$, defined for $z\not=0$,
satisfies $\Pi^\#(\frac{1}{\hbar}d \log
|z|^2)=\frac{i}{\hbar}z\del_z+c.c.$ there. So, on the open
symplectic leaf $\C_0$, the action of $\T$ is Hamiltonian. Since
$\Pi$ is non-degenerate on $\C_0$, any other Hamiltonian function
must differ from $\frac{1}{\hbar}\log |z|^2$ by a constant there and
since none of these functions extend to the origin, it follows that
the action of $\T$ on $(\C,\Pi)$ is Poisson but not Hamiltonian.
However, its restriction to $\C_0$ is Hamiltonian.

The action of $\T_\C/\T$ is generated infinitesimally by the radial
field $z\del_z+c.c.$  A Hamiltonian function with respect to $\Pi$
for this field is $\frac{1}{\hbar}\mathrm{Arg}(z)$ which can be
defined on $\C\setminus (-\infty,0]$.  Of course, this function can
not even be extended to $\C_0$, let alone to all of $\C$.  So the
action of $\T_\C/\T$ is not Hamiltonian on $(\C_0,\Pi)$ or
$(\C,\Pi)$.

For general $d$ consider the leaf $(\C_0)^S$ indexed by an arbitrary
$S\subset \{1,2,\ldots,d\}$.  By the arguments above, it is clear
that the action of $\T_\C^d/\T_\C^{S^c}$ on $(\C_0)^S$ is Poisson
but not globally Hamiltonian.  However, one can define
\begin{equation}
\Phi_S(z)=\sum_{s\in S}\frac{1}{\hbar}\log |z_s|^2\epsilon_s,
\end{equation}
on $(\C_0)^S$.  Then $\Phi_S\colon (\C_0)^S\to (\R^d)^*$ is a
momentum mapping for the action of the compact sub-torus $\T^S$ on
$(\C_0)^S$.
\end{proof}

Let $T$ be a compact abelian Lie group, let $\Delta$ be a Delzant
polytope in $\t^*$, and let $\Sigma$ be its dual fan in $\t$. Since
$\Pi$ on $\C^d$ is invariant under the action of $\T_\C^d$, so is
its restriction to $\mathcal U_\Sigma$ and thus $\Pi$ descends to
the toric variety $X(\Sigma)=\mathcal U_\Sigma/N_\C$.
\begin{notation}
Let $\Pi_\Sigma$ denote the Poisson structure on $X(\Sigma)$
coinduced by the quotient map $q\colon (\mathcal U_\Sigma,\Pi) \to
\mathcal U_\Sigma/N_\C=X(\Sigma)$.
\end{notation}

Given that the symplectic leaves of $(\mathcal U_\Sigma,\Pi)$ are
$\T_\C^d$-orbits, the following theorem is immediate given the
results of section \ref{Construction_of_Varieties}.

\begin{theorem}
The symplectic leaves of $(X(\Sigma),\Pi_\Sigma)$ are in bijection
with the set of cones in $\Sigma$, or equivalently, with the set of
faces of the polytope $\Delta$.
\end{theorem}

The next theorem shows how to construct $\Pi_\Sigma$ locally and the
following notation will arise and be used throughout the rest of the
paper.

\begin{notation}\label{B_VNotation}
Given a vertex $V$ of $\Delta$ labeled by the set
$S_V=\{s_1,\ldots,s_n\}$ where $s_1<s_2<\dots<s_n$. Let
$s_{n+1}<\dots<s_d$ denote the elements of $S_V^c$.  Let $A_V$
denote the $n\times (d-n)$ integral matrix such that
$u_{s_{n+\ell}}+\sum_{k=1}^n u_{s_k}A_{k\ell}=0$. Then
$B_V=1+A_V(A_V)^t$.
\end{notation}

\begin{theorem}\label{local_representation}
Let $V$ be a vertex of $\Delta$.  This determines a $n\times (d-n)$
integral matrix $A_V$ and an associated coordinate chart $w_V\colon
zN_\C\mapsto (w_1,\ldots,w_n)$.  In terms of these coordinates the
Poisson structure $\Pi_\Sigma$ has the form
\begin{equation}
\Pi_\Sigma  =  \sum_{p,q=1}^n iB_{pq} \wbar_p w_q
\del_{\wbar_p}\wedge \del_{w_q}
\end{equation}
where $B_{pq}$ are the components of the symmetric positive definite
integral matrix $B_{V}=1+A_{V}(A_{V})^t$.  In particular,
$\Pi_\Sigma$ is a homogeneous quadratic Poisson structure in each
vertex chart.
\end{theorem}
\begin{proof}
Again, it suffices to consider the case where
$S_\nu=\{1,2,\ldots,n\}$.  Now for each $s=1,2,\ldots,n$,
\[
z_r\frac{\del w_s}{\del
z_r}=z_r\widetilde{z}^{-a_s}\delta_{rs}\text{ if
 }r=1,2,\ldots,n
\]
 whereas
\[
z_r\frac{\del w_s}{\del_{z_r}}=-a_{sr}z_s
\widetilde{z}^{-a_s}=-a_{sm}w_s
\]
for each $r=n+m$, where  $m=1,2,\ldots,d-n$.
  Thus, under the quotient map
$z_r\del_{z_r}\mapsto w_r\del_{w_r}$ for each $r=1,2,\ldots,n$
whereas $z_r\del_{z_r}\mapsto -\sum_{s=1}^n a_{sm}w_s\del_{w_s}$ for
each $r=n+m$ where $m=1,2,\ldots, d-n$.  Hence, $\Pi$ maps to
\begin{eqnarray*}
\Pi_\Sigma & = & \sum_{r=1}^n i|w_r|^2\del_{\wbar_r}\wedge \del_{w_r} + \sum_{m=1}^{d-n}\left(\sum_{s=1}^n a_{sm}\wbar_s\del_{\wbar_s}\right)\wedge\left(\sum_{t=1}^n a_{tm}w_t\del_{w_t}\right)\\
           & = & \sum_{\ell=1}^n \left(1+\sum_{m=1}^{d-n}a_{\ell m}a_{\ell m}\right) i|w_\ell|^2
\del_{\wbar_\ell}\wedge \del_{w_\ell} \\
& &  + \sum_{k<\ell} \left(\sum_{m=1}^{d-n} a_{km}a_{\ell m}\right)
i\wbar_k w_\ell
           \del_{\wbar_k}\wedge \del_{w_\ell} + c.c.
\end{eqnarray*}
The coefficients are clearly the components of the matrix
$B_{V}=1+A_{V}(A_{V})^t$ as claimed.
\end{proof}

\begin{corollary}\label{TopWedgePowerPiSigma}
Let $V$ be a vertex of $\Delta$. In terms of the vertex coordinates
$w_1,\ldots,w_n$,
\begin{equation}
\textstyle{\frac{1}{n!}}\Pi_\Sigma^n = i^n \det(B_{V})|w_1|^2\dots
|w_n|^2\,\del_{\wbar_1}\wedge \del_{w_1}\wedge\ldots \wedge
\del_{\wbar_n}\wedge \del_{w_n}.
\end{equation}
\end{corollary}
\begin{proof}
Let $\tilde{\pi}=\sum_{p,q=1}^n B_{V}^{pq} i\del_{\zbar_p}\wedge
\del_{z_q}$ on $\C^n$.  Then $\tilde{\pi}\mapsto \Pi_\Sigma$ under
the map $z\mapsto \exp(z)=w$.  Furthermore,
\[
\textstyle{\frac{1}{n!}}\tilde{\pi}^n=i^n\det(B_V)
\,\del_{\zbar_1}\wedge\del_{z_1}\wedge\dots \wedge
\del_{\zbar_n}\wedge \del_{z_n}
\]
and therefore
\begin{eqnarray*}
\textstyle{\frac{1}{n!}}\Pi_\Sigma^n & = & \textstyle{\frac{1}{n!}}(\exp_*\tilde{\pi})^n \\
& = & \exp_*(\textstyle{\frac{1}{n!}}\tilde{\pi}^n) \\
& = & i^n \det (B_{V})|w_1|^2\dots
|w_n|^2\,\del_{\wbar_1}\wedge\del_{w_1}\wedge\dots \wedge
\del_{\wbar_n}\wedge \del_{w_n}
\end{eqnarray*}
as was to be shown.
\end{proof}

\begin{theorem}\label{MomentumMaponLeaves}
The action of $T_\C$ on $(X(\Sigma),\Pi_\Sigma)$ is Poisson but not
Hamiltonian.  However, each symplectic leaf of
$(X(\Sigma),\Pi_\Sigma)$ has a Hamiltonian action by a sub-torus of
$T$.
\end{theorem}
\begin{proof}
Consider the $T_\C$-orbit in $X(\Sigma)$ corresponding to a face $F$
of $\Delta$, i.e., $\Phi_\Delta^{-1}(F^\circ)$.  In assuming that
$F$ is labeled by $\{1,2,\ldots,s\}$ where $s<n$, no generality is
lost. Then this orbit is contained in the open set $\mathcal
U_V/N_\C$ where $V$ is the vertex labeled by $\{1,2,\ldots,n\}$ and
is characterized in the coordinates by the inequality $w_1\ldots
w_s\not=0$ together with the equations $w_{s+1}=0,\ldots,w_n=0$.  It
thus suffices to define the momentum map for the action of
$T_F:=\exp(\mathrm{span}\{u_{s+1},\ldots,u_{n}\})$ on this leaf in
local coordinates.

Let $\eta_1,...,\eta_n$ be the frame for $\t^*$ dual to the frame
$u_1,\ldots,u_n$ for $\t$.  Let $B_F$ denote the principal $s\times
s$ block of the matrix $B_V$ and let $b_F$ denote its inverse, this
exists because $B_V$ is positive definite and therefore each of its
principal minors is non-zero.  Define $\psi_F\colon
\Phi_\Delta^{-1}(F^\circ)\to \t^*$ by
\begin{equation}
\psi_F(w_1,\ldots,w_s)=(h/\hbar)\sum_{k,\ell=1}^s b_{k
\ell}\log|w_\ell|^2\eta_k.
\end{equation}
Then $d\langle \psi_F,u_m\rangle = \sum_{\ell=1}^s
b_{m\ell}d\log|w_\ell|^2$ and thus
\begin{eqnarray*}
\Pi_\Sigma^\#(d\langle \psi_F,u_m\rangle) & = &
(h/\hbar)\sum_{p,q=1}^s iB_{pq}\wbar_pw_q \del_{\wbar_p}\wedge
\del_q^\#\left(\sum_{\ell=1}b_{m\ell}\log |w_\ell|^2\right) \\
& = & (h/\hbar)\sum_{p,q=1}^n iB_{pq}b_{mp}w_q\del_{w_q}+c.c. \\
& = & (h/\hbar)iw_m\del_{w_m}+c.c.
\end{eqnarray*}
This shows that the action of $T_F$ on $\Phi_\Delta^{-1}(F^\circ)$
is Hamiltonian.  In particular, the action of $T$ on the open
symplectic leaf is Hamiltonian, but clearly it does not extend to
all of $X(\Sigma)$.  This completes the proof.
\end{proof}


\section{The modular class and $H^1(X(\Sigma),\Pi_\Sigma)$\label{TheModularClass}}

If $(M,\pi_M)$ is a Poisson manifold, then $d_\pi\colon \mathcal
V^k(M)\to \mathcal V^{k+1}(M)$ defined by $d_{\pi_M}(W)=[W,\pi_M]$
for each $W\in \mathcal V^k(M)$ is a differential, where
$[\cdot,\cdot]$ is the Schouten bracket of multi-vector fields.  The
algebra of multi-vector fields $\mathcal V(M)$ together with
$d_{\pi_M}$ forms a complex.  The homology $H^\bullet (M,\pi_M)$ of
this complex is called the Poisson cohomology of $M$.  The anchor
map $\pi^\#\colon T^*M\to TM$ induces a chain map from the de Rham
complex to the Poisson complex and therefore induces a map
$\pi^\#\colon H_{DR}^\bullet(M)\to H^\bullet(M,\pi_M)$. When $\pi_M$
is non-degenerate on $M$, the induced map is an isomorphism.  So for
symplectic manifolds, the Poisson cohomology is finite dimensional.
At the other extreme, every manifold admits the zero poisson
structure $\pi_M=0$ and in this case $H^\bullet(M,\pi_M)=\mathcal
V^\bullet (M)$ which is infinite dimensional.  For structures in
between these two extremes, the Poisson cohomology can be quite
difficult to compute because of the lack of a powerful method for
computation.

Every Poisson manifold $(M,\pi_M)$ has a distinguished class in the
first Poisson cohomology called the modular class (cf.
\cite{Weinstein}).  A vector field representing this class may be
computed as follows. Choose a smooth positive density $\mu$ on $M$.
For each smooth function $f$, the Lie derivative of $\mu$ with
respect to the Hamiltonian vector field $\pi^\#(df)$ is again a
density and thus can be expressed uniquely as $\theta_\mu(f)\mu$
where $\theta_\mu(f)$ is a smooth function. Due to the skew-symmetry
of $\pi_M$, the assignment $f\mapsto \theta_\mu(f)$ turns out to be
a derivation on smooth functions and hence defines a vector field
$\theta_\mu$ called the modular vector field of $\pi_M$ associated
with the density $\mu$. If $a$ is a positive function on $M$, then
$a\mu$ is a positive density and
$\theta_{a\mu}=\theta_\mu-\pi^\#(d\log a)$. In the case that
$(M,\pi_M)$ is a non-degenerate Poisson manifold and
$\mu=\frac{1}{n!}\pi_M^{-n}$ is the Liouville volume form of the
symplectic structure, $\theta_\mu=0$.

\begin{theorem}\label{modular_class_nonzero}
Let $(M,\pi_M)$ be an orientable Poisson manifold with the property
that $\pi_M$ is non-degenerate on a open dense subset of $M$ but is
not regular. Then the modular class of $(M,\pi_M)$ is non-trivial.
\end{theorem}
\begin{proof}
Suppose $\dim M=2n$ and let $\mu$ be a nowhere vanishing $2n$-form
on $M$.  Let $\hat{\mu}$ be the nowhere vanishing $2n$-vector field
dual to $\mu$ in the sense that $\langle \mu,\hat{\mu}\rangle=1$ at
each point of $M$. Let $U$ be the union of the open symplectic
leaves of $\pi_M$. Then $\frac{1}{n!}\pi_M^n=a\hat{\mu}$ where $a\in
C^\infty(M)$ is a function vanishing on the complement of $U$. On
the open symplectic leaves, $a\hat{\mu}$ is dual to the $2n$-form
$\frac{1}{a}\mu$ which is the Liouville volume form of $\pi_M$ on
the open symplectic leaves. Thus
\begin{eqnarray*}
\theta_{\frac{1}{a}\mu} & = & \theta_\mu-\pi_M^\#(d\log(1/a))\\
 & = & \theta_\mu+\pi_M^\#(d\log(a))
\end{eqnarray*}
on $U$ which implies that $\theta_\mu=-\pi_M^\#(d\log(a))$ on $U$ as
$\theta_{\frac{1}{a}\mu}=0$. Any other primitive of $\theta_\mu$ on
$U$ must differ from $d\log(a)$ by a locally constant function and
none of these extend to the boundary of $U$ as $a=0$ there.  Thus,
$[\theta_\mu]\not=0$ in $H^1(M,\pi_M)$.
\end{proof}
\begin{corollary}
The modular class of $(X(\Sigma),\Pi_\Sigma)$ is non-zero.
\end{corollary}
\begin{proof}
The toric variety $X(\Sigma)$ is a complex manifold.  Thus it is
oriented.  As the symplectic leaves of $\Pi_\Sigma$ are the
$T_\C$-orbits in $X(\Sigma)$, $\Pi_\Sigma$ is not regular and yet is
non-degenerate on an open dense set.
\end{proof}

A consequence of the proof of Theorem \ref{MomentumMaponLeaves} is
that there is a linear injection $\t+i\t\to
H^1(X(\Sigma),\Pi_\Sigma)$ because the action of $T_\C$ is Poisson
and yet fails to be Hamiltonian because the momentum map for the
action on the open symplectic leaf fails to extend.  It is natural
to ask whether the modular class is disjoint from this image and
this question is answered by the next result.  For brevity, the
following notation is introduced.

\begin{notation}
If $u_1,\ldots,u_n$ is the basis of $\Lambda$ associated to given
vertex of $\Delta$, write $R_1,\ldots,R_n$ and $D_1,\ldots,D_n$ for
the vector fields on $X(\Sigma)$ corresponding to $u_1,\ldots,u_n$
and $iu_1,\ldots,iu_n$, respectively.  In terms of the local
coordinates $w_1,\ldots,w_n$ associated to the same vertex,
$R_k=(h/\hbar)iw_k\del_{w_k}+c.c.$ and
$D_k=(h/\hbar)w_k\del_{w_k}+c.c.$ for each $k=1,2,\ldots,n$.
\end{notation}

As the action of $\T_\C$ is Poisson, the vector fields $R_k$ and
$D_k$ determine classes in $H^1(X(\Sigma),\Pi_\Sigma)$ for each $k$
and a $\R$-linear map $\t+i\t\to H^1(X(\Sigma),\Pi_\Sigma)$.

\begin{theorem}
For each vertex of $\Delta$, the associated classes
\[[R_1],\ldots,[R_n],[D_1],\ldots,[D_n],[\theta_\mu]\in
H^1(X(\Sigma),\Pi_\Sigma)\] are linearly independent.  Hence $\dim
H^1(X(\Sigma),\Pi_\Sigma)\ge 2n+1$.
\end{theorem}
\begin{proof}
Let $w_1,\ldots,w_n$ be the associated vertex coordinates.  Let
\[\lambda=i^n dw_1\wedge d\wbar_1\wedge\dots \wedge dw_n\wedge
d\wbar_n\]
and let $\frac{1}{n!}\Omega_\Sigma^n$ denote the
Liouville volume form for $\Pi_\Sigma$ on the symplectic leaf
$\{w_1\dots w_n\not=0\}$.  Then $\lambda=a
\frac{1}{n!}\Omega_\Sigma^n$ where $a=\det(B_V)|w_1|^2\dots
|w_n|^2$.

Now $d\log a=\sum_{\ell=1}^n d\log |w_\ell|^2$ and therefore
\begin{eqnarray*}
\theta_\lambda & = & -\Pi_\Sigma^\#(d\log a) \\
& = & -\sum_{\ell,p,q=1}^n iB_{pq}\wbar_p w_q
\frac{\delta_{p\ell}}{\wbar_\ell}+c.c. \\
& = & -\sum_{p,q=1}^n iB_{pq}w_q\del_{w_q}+c.c. \\
& = & -\textstyle{\frac{1}{(h/\hbar)}}\sum_{p,q=1}^n iB_{pq}R_q.
\end{eqnarray*}
So, as one would expect, $\theta_\lambda$ is Hamiltonian on the
leaf.  But more importantly, it lies the $\R$-span of the vector
fields $R_1,\ldots,R_n$ and thus defines a vector field on all of
$X(\Sigma)$.

If $\mu$ is a positive volume form on $X(\Sigma)$ then
$\mu=f\lambda$ in the chart where $f$ is a smooth positive function.
Thus $\theta_\mu=\theta_\lambda-\Pi_\Sigma^\#(d\log f)$ on $\mathcal
U_V/N_\C$.  Since $X(\Sigma)$ is compact, $f$ must vanish at the
boundary of the chart and thus every primitive of
$\theta_\mu-\theta_\lambda$ on $\mathcal U_V/N_C$ fails to extend to
all of $X(\Sigma)$.  Assume that $\mu$ is $T$-invariant (such a
volume form can be constructed by averaging an arbitrary one over
the action of $T$).  Then $\theta_\mu$ must lie in the
$C^\infty(X(\Sigma))^T$ span of $R_1,\ldots,R_n$ because
$\Pi_\Sigma$ and $\mu$ are $T$-invariant. But as
$\theta_\mu-\theta_\lambda$ is not Hamiltonian, it follows that the
modular class $[\theta_\mu]$ is independent of the classes
$[R_1],\ldots,[R_n]$ and $[D_1],\ldots,[D_n]$.
\end{proof}


\section{Poisson cohomology of $(\CP^1,\Pi_\Sigma)$.\label{PoissonCohomforCP1}}

Let $D=z\del_z+c.c.$ and $R=iz\del_z+c.c.$ be the dilation and
rotation vector fields, respectively, on $\C$.  Then
$\Pi=i|z|^2\del_{\zbar}\wedge \del_z=\frac{1}{2}D\wedge R$ in
dimension $d=1$. The restriction of $\Pi$ to $\C_0$ is symplectic
and thus $H^\bullet(\C_0,\Pi)=H^\bullet_{DR}(\C_0)$ as a graded
vector space. On all of $\C$, however, the Poisson cohomology is
more complicated. This was worked out first in \cite{Nakanishi}.

\begin{theorem}[Nakanishi]
The Poisson cohomology of $(\C,\Pi)$ is finite dimensional.  In
terms of generators,
\begin{enumerate}
\item[i)] $H^0(\C,\Pi)=\R[1]$,
\item[ii)] $H^1(\C,\Pi)=\R[D] + \R[R]$, and
\item[iii)] $H^2(\C,\Pi)=\R[\pi] + \R[\Pi]$.
\end{enumerate}
\end{theorem}

Using these results, one can compute the Poisson cohomology of the
simplest toric variety via a Mayer-Vietoris argument.  Let $T$ be a
group isomorphic to $\T$.  The convex hull of any two distinct
points in $\t^*$ is then a Delzant polytope.  Let $\eta_1$ be a
generator of $\Lambda^*$, let $u_1$ denote the dual generator of
$\Lambda$, and let $u_2=-u_1$.  Let $\Delta$ be the convex hull of
$\{-\frac{1}{2}\eta_1,\frac{1}{2}\eta_1\}$ in $\t^*$.  The facets of
$\Delta$ are the two vertices of $\Delta$, since $\Delta$ has
dimension 1. Label $-\frac{1}{2}\eta_1$ as facet 1 and the other
vertex as facet 2 corresponding to the enumeration of the generators
of $\Sigma^{(1)}$, $u_1,u_2$. Then $N_\C$ is the diagonal copy of
$\T_\C$ in $\T_\C^2$. Every subset of $\{1,2\}$ labels a face of
$\Delta$ except the set $\{1,2\}$ itself. Thus, $\mathcal
U_\Sigma=\C^2\setminus \{(0,0)\}$ and the quotient
$X(\Sigma)=\mathcal U_\Sigma/N_\C$ is $\CP^1$.

Let $D_1$ denote the vector field corresponding to $iu_1\in \t_\C$.
Then $D_1$ and $R_1$ are commuting vector fields on $\CP^1$ which
vanish only at the fixed points of the $\T_\C$ action, i.e., points
$\Phi_\Delta^{-1}(V)$ for each vertex $V$ of $\Delta$.

\begin{theorem}\label{Poisson_Cohomology_CP1}
The Poisson cohomology of $(\CP^1,\Pi_\Sigma)$ is finite
dimensional. As real vector spaces
\begin{enumerate}
\item[i)] $H^0(\CP^1,\Pi_\Sigma)=\R$,
\item[ii)] $H^1(\CP^1,\Pi_\Sigma)=\R^3$, and
\item[iii)] $H^2(\CP^1,\Pi_\Sigma)=\R^4$.
\end{enumerate}
\end{theorem}
\begin{proof}
The proof applies the Mayer-Vietoris sequence for Poisson cohomology
to the affine open cover of $\CP^1$ by the open sets $U=\mathcal
U_{\{1\}}/N_\C$ and $V=\mathcal U_{\{2\}}/N_\C$ associated to the
vertices of $\Delta$.  Let $w$ denote the coordinate in the first
chart and $w'$ the coordinate in the second chart.  In terms of $w$,
$\Pi_\Sigma=2(h/\hbar)i|w|^2\del_{\wbar}\wedge \del_w$ and the
expression is identical in terms of $w'$ where $w'=1/w$ on the
overlap.  The intersection $U\cap V$ is the open symplectic leaf of
$\Pi_\Sigma$, thus $H^\bullet(U\cap V,\Pi_\Sigma)\simeq
H^\bullet_{DR}(U\cap V)$.  As $\Pi_\Sigma$ is non-degenerate on an
open dense set the anchor map induces an isomorphism
$H^0(\CP^1,\Pi_\Sigma)\simeq H^0_{DR}(\CP^1)$.  Thus, the initial
row of the Mayer-Vietoris sequence is exact and the interesting part
begins with $H^1(\CP^1,\Pi_\Sigma)$.

By the theorems of Nakanishi the terms $H^\bullet(U,\Pi_\Sigma)$ are
known.  The isomorphism $H^\bullet(U\cap V,\Pi_\Sigma)\simeq
H^\bullet_{DR}(\T_\C)\simeq H^\bullet_{DR}(S^1)=0$ gives that
$H^1(U\cap V,\Pi_\Sigma)=\R$ and $H^2(U\cap V,\Pi_\Sigma)=0$. Hence,
the interesting part of the Mayer-Vietoris sequence takes the form
\[
0\to H^1(\CP^1,\Pi_\Sigma) \to \R^2\oplus \R^2 \to \R
\xrightarrow{\delta} H^2(\CP^1,\Pi_\Sigma)\to \R^2\oplus \R^2 \to 0
\]
where $\delta\colon H^1(U\cap V,\Pi_\Sigma)\to
H^2(\CP^1,\Pi_\Sigma)$ is the connecting homomorphism.  From this
exact sequence, it is clear that the Poisson cohomology of
$(\CP^1,\Pi_\Sigma)$ is finite dimensional as $\dim H^1+1=\dim H^2$
and $\dim H^1\le 4$. A generator of $H^1_{DR}(U\cap V)$ is
$d\mathrm{Arg}(w)$, and under the anchor map
$\Pi_\Sigma^\#(d\mathrm{Arg}(z))=\frac{2}{(\hbar/h)}D_1$ which
extends globally to all of $\CP^1$.  Therefore, the connecting
homomorphism $\delta \colon H^1(U\cap V,\Pi_\Sigma)\to
H^2(\CP^1,\Pi_\Sigma)$ is the zero map, so
$H^1(\CP^1,\Pi_\Sigma)\simeq \R^3$ and $H^2(\CP^1,\Pi_\Sigma)=\R^4$.
The proof is complete.
\end{proof}

It is interesting to note that the $T$-invariant generators of
$H^2(\CP^1,\Pi_\Sigma)$ are not all algebraic even though
$\Pi_\Sigma$ is real algebraic.  The space of $T$-invariant smooth
algebraic bi-vector fields on $\CP^1$ is only 3-dimensional as each
is of the form
\[i(a+b|w|^2+c|w|^4)\del_{\wbar}\wedge \del_{w}\]
expressed in the coordinate associated to vertex 1 where
$a,b,c\in\R$. The bi-vectors $i\del_{\wbar}\wedge\del_w$,
$i|w|^2\del_{\wbar}\wedge \del_w$, and $i|w|^4\del_{\wbar}\wedge
\del_w$, which vanish at fixed point $2$ but not $1$, vanish at $1$
and $2$, and vanish $1$ but not $2$, respectively, form a basis for
this space. In their span, with $a=1$, $b=2$, and $c=1$ is the
non-degenerate Poisson structure
$\pi_\Delta=i(1+|w|^2)^2\del_{\wbar}\wedge\del_w$ arising the
Delzant construction.  A fourth independent generator is the real
analytic bi-vector $e^{-|w|^2}\del_{\wbar}\wedge\del_w$ which
vanishes only at vertex $2$ but with infinite order.


\section{The modular vector field on $(\CP^n,\Pi_\Sigma)$.\label{ModVectorFieldonCPn}}

Let $T$ be a Lie group isomorphic to $\T^n$ and let $u_1,\ldots,u_n$
be a basis $\Lambda\subset \t$ and set $u_{n+1}=-\sum_{\ell=1}^n
u_\ell$.  Let $\eta_1,\ldots,\eta_n$ be the dual basis of
$\Lambda^*\subset \t^*$.  Let $\Delta'$ be the convex hull of the
points $\{0,\eta_1,\ldots,\eta_n\}$ in $\t^*$. Then $\Delta'$ is a
standard $n$-simplex and a Delzant polytope and its dual fan
$\Sigma$ is the data needed to construct $\CP^n$ as a toric variety.
Every other polytope whose dual fan in $\Sigma$ differs from
$\Delta'$ only by a translation or overall change of scale.  The
translate $\Delta=\Delta'-\nu$ where
$\nu=\frac{1}{n+1}(\eta_1+\ldots+\eta_n)$ is a Delzant polytope
containing the origin as an interior point.  The K\"ahler form
$\omega_\Delta$ it determines on $X(\Sigma)=\CP^n$ is the one
relevant for the discussion in this section. Define the map $p\colon
(h\Z)^d\to \Lambda$ by the assignments $he_\ell\mapsto u_\ell$, for
each $\ell=1,2,\ldots,n+1$. The kernel
$\n=\mathrm{span}_\R\{hc=he_1+\ldots+he_n\}$, thus $N_\C$ is the
standard diagonal embedding of $\T_\C$ in $\T_\C^{n+1}$.  For a
simplex, every subset of $\{1,2,\ldots,n+1\}$ labels a face of
$\Delta$ except the set itself.  Thus $\mathcal
U_\Sigma=\C^{n+1}\setminus \{0\}$ and since the orbits of $N_\C$ on
this set are the intersections of complex lines through the origin
with $\mathcal U_\Sigma$, it follows that $X(\Sigma)=\CP^n$ as
desired.  With
$\lambda=\frac{1}{n+1}(\epsilon_1+\ldots+\epsilon_{n+1})$, one has
\[
\Delta = \{\nu\in \t^*\colon hp^*(\nu)+\lambda\ge 0\}.
\]

Consider the vertex $V$ of $\Delta$ labeled by
$S_V=\{1,2,\ldots,n\}$.  The open set $\mathcal U_V/N_\C$ is the
open set of classes $[z]$ where $z_{n+1}\not=0$.  Given
$[z]=[z_1:\ldots :z_{n+1}]$ in $\CP^n$ with $z_{n+1}\not=0$, the
associated affine coordinates are
\begin{equation}
w_1=z_1z_{n+1}^{-1},\ldots, w_n=z_nz_{n+1}^{-1}
\end{equation}
as the matrix $A_V$ is the column vector whose transpose is
$(1,\ldots,1)$. Let $\phi$ be the generator of $\n^*$ dual to the
generator $c$ of $\n$ and linearly identify $\n^*$ with $\R$ via the
assignment $\phi\mapsto 1$.  Then $t\phi\mapsto t$ gives a global
coordinate $\mu(\mathcal U_\Sigma)\to \R_{>0}$ and the
diffeomorphism $\mathcal U_\Sigma/N \to X(\Sigma)\times \mu(\mathcal
U_\Sigma)$ is expressed in terms of the coordinates
$(w_1,\ldots,w_n,t)$ for $\mathcal U_{S_V}/N_\C\times \mu(\mathcal
U_\Sigma)$ by $w$ as above and
\[
t=\langle \mu,c\rangle = \frac{1}{\hbar}|z_1|^2+\ldots
+\frac{1}{\hbar}|z_{n+1}|^2.
\]
In these coordinates, the image of $h\pi$ under the quotient map
lands in the sub-bundle $\bigwedge^2 T \C^n$ of $\bigwedge^2
T(\C^n\times \R_+)$, so in doing computations it suffices to
calculate only these components.  Under the quotient map
\[
\del_{z_k}\mapsto z_{n+1}^{-1}\del_{w_k}
\]
plus terms in $T\R_+$ for $k=1,2,\ldots, n$ and
\[
\del_{z_{n+1}}\mapsto \sum_{p=1}^n (-z_p)z_{n+1}^{-2}\del_{w_p}
\]
plus terms in $T\R_+$. Therefore,
\begin{eqnarray*}
h\pi & \mapsto &  \pi_{t\Delta}=hi\sum_{k=1}^n (|z_{n+1}|^2)^{-1}
\left(1+|z_k|^2|z_{n+1}|^{-2}\right) \del_{\wbar_k}\wedge \del_{w_k}
\\ & & \,\, + hi \sum_{k<\ell} \zbar_k z_\ell (|z_{n+1}|^2)^{-2}
\del_{\wbar_k}\wedge \del_{w_\ell} + c.c.
\end{eqnarray*}
It remains to determine the functions $z_1,\ldots,z_{n+1}$ in terms
of $w=(w_1,\ldots,w_n)$ and $t$.  Write $\|w\|^2$ for the sum
$|w_1|^2+\ldots+|w_n|^2$.  The point $(w,t)\in \C^n\times \R_+$
corresponds to the unique $N$-orbit in $\mathcal U_\Sigma$ which
passes through the point $(w_1F,\ldots,w_nF,F)$ where $F=F(w)$ is a
positive real number such that $t=\mu(w_1F,\ldots,w_nF,F)$.  In
other words
\[
t=\textstyle{\frac{1}{\hbar}}|w_1|^2 F^2+\ldots
\textstyle{\frac{1}{\hbar}}|w_n|^2F^2+\textstyle{\frac{1}{\hbar}}F^2,
\]
or equivalently, $F=\sqrt{t\hbar /(1+\|w\|^2)}$.  Substituting
$z_\ell=w_\ell F$ for $\ell=1,2,\ldots,n$ and $z_{n+1}=F$, gives
\begin{eqnarray}
\pi_{t\Delta}  & = & \frac{i}{(\hbar/h)t}\sum_{k=1}^n
(1+\|w\|^2)(1+|w_k|^2)\del_{\wbar_k}\wedge \del_{w_k} \label{hpiSigmaLocalCPn}\\
& & + \frac{i}{(\hbar/h)t} \sum_{k<\ell} w_k w_\ell
(1+\|w\|^2)\del_{\wbar_k}\wedge\del_{w_k} +c.c.\nonumber
\end{eqnarray}
Note that it is independent of the parameter $h$ as $h/\hbar$ is the
circumference of the unit circle.  The expression for $\pi_\Delta$
is then obtained by setting $t=1$ in (\ref{hpiSigmaLocalCPn}).

The components of the momentum map in terms of the basis
$u_1,\ldots,u_n$ are easily found to be
\begin{equation}\label{CPnMomentumComponents}
\langle
\Phi_\Delta(w),u_\ell\rangle=\frac{(n+1)|w_\ell|^2-1-\|w\|^2}{(n+1)(1+\|w\|^2)}.
\end{equation}
To compute this, observe that the map $hp^*$ sends $u_\ell\mapsto
\epsilon_\ell-\epsilon_{n+1}$ for each $\ell=1,\ldots,n$.  From the
relation
\[
hp^*(\Phi_\Delta(w))+\lambda=\sum_{\ell=1}^{n+1}
\textstyle{\frac{1}{\hbar}}|z_\ell|^2\epsilon_\ell
\]
one can solve for the components of the momentum map $\Phi_\Delta$.
Applying both sides to $e_\ell$ gives
\begin{equation}\label{MomentumRelation}
\langle
\Phi_\Delta(w),u_\ell\rangle+\textstyle{\frac{1}{n+1}}=\frac{|w_\ell|^2}{1+\|w\|^2}\text{
for }\ell=1,2,\ldots, n
\end{equation}
from which the equations
above follow.

The local formulas for $\Pi_\Sigma$ can be obtained using Theorem
\ref{local_representation}.  Here the matrix $B_V=1+A_VA_{V}^t$ is
\[
B_V=\begin{pmatrix} 2 & 1 & \dots  & 1 \\
                        1 & 2 &        & \vdots  \\
                  \vdots  &   & \ddots &   \\
                        1 & \dots &    & 2 \\
\end{pmatrix}
\]
whence
\begin{equation}\label{PiSigmaLocalCPn}
\Pi_\Sigma=\sum_{k=1}^n i2|w_k|^2\del_{\wbar_k}\wedge \del_{w_k} +
\sum_{k<\ell} i \wbar_k w_\ell \del_{\wbar_k}\wedge \del_{w_\ell} +
c.c.
\end{equation}
Moreover, by Corollary \ref{TopWedgePowerPiSigma},
\begin{equation}\label{PiSigmaTopPowerLocalCPn}
\textstyle{\frac{1}{n!}}\Pi_\Sigma^n = i^n(n+1)  |w_1|^2\dots
|w_n|^2 \del_{\wbar_1}\wedge \del_{w_1}\wedge\dots \wedge
\del_{\wbar_n}\wedge \del_{w_n}
\end{equation}
as the determinant of the matrix $B_V$ is $(n+1)$ (it is of the form
identity plus a rank one matrix whose non-zero eigenvalue is $n$).

\begin{prop}
In terms of the local coordinates $w_1,\ldots,w_n$,
\begin{equation}\label{PiDeltaTopPowerLocalCPn}
\frac{1}{n!}\pi_\Delta^n = \frac{i^n}{(\hbar/h)^n}(1+\|w\|^2)^{n+1}
\del_{\wbar_1}\wedge \del_{w_1} \wedge \dots \wedge
\del_{\wbar_n}\wedge \del_{w_n}.
\end{equation}
\end{prop}
\begin{proof}
Regard $w=(w_1,\ldots,w_n)$ as a column vector. From
(\ref{PiSigmaLocalCPn}) we see that $\pi_\Delta$ has a local
expression of the form $\frac{i}{(\hbar/h)}(1+\|w\|^2)\xi$ where
$\xi$ is a bi-vector field on $\C^n$ of pure bi-degree $(1,1)$.
Thus, one can use the identity
$\pi_\Delta^n=\frac{i^n}{(\hbar/h)^n}(1+\|w\|^2)^n\xi^n$ to simplify
the calculation. In terms of the basis
$d\wbar_1,\ldots,d\wbar_n,dw_1,\ldots,dw_n$, $\xi$ is represented by
a matrix $X$ of the form
\[
X=\begin{pmatrix}    0 & X' \\
    -X' & 0 \\\end{pmatrix}
\]
relative to the splitting of the cotangent space into $(0,1)$ and
$(1,0)$ parts where $X'=1+ww^*$. Thus,
\[
\xi^n=n!\,\mathrm{Pfaff}(X)\,\del_{\wbar_1}\wedge \del_{w_1}\wedge
\ldots \wedge \del_{\wbar_n}\wedge \del_{w_n}.
\]
From the block form of $X$, it follows that
$\mathrm{Pfaff(X)}=\det(X')$.

The operator $ww^*$ has rank one when $w\not=0$ and a generator of
the non-zero eigenspace is the vector $w$ with associated eigenvalue
$\|w\|^2$.  Thus, $\det(X')=\det(1+ww^*)=(1+\|w\|^2)$ and this
relationship holds as $w\to 0$.  The result follows immediately.
\end{proof}

\begin{theorem}Let $V$ be a vertex of $\Delta$.  For convenience, assume that it is labeled by $\{1,2,\ldots,n\}$.
Let $R_1,\dots,R_n$ denote the vector fields on $\CP^n$
corresponding to the vectors $u_1,\ldots,u_n\in \t$ and write $u$
for the frame $u=(u_1,\ldots,u_n)$. Then the modular vector field
$\theta_\mu$ of $\Pi_\Sigma$ associated to the Delzant volume form
$\mu=\textstyle{\frac{1}{n!}}\omega_\Delta^n$ has the form
\begin{equation}\label{ModularVectorFieldonCPn}
\theta_\mu=\frac{(n+1)}{h/\hbar}\sum_{\ell=1}^n \langle
\Phi_\Delta,(uB_{V})_\ell\rangle \, R_\ell
\end{equation}
where $B_V$ is the integral matrix attached to $V$.
\end{theorem}
\begin{proof}As $\theta_\mu$ is smooth, it suffices to establish
(\ref{ModularVectorFieldonCPn}) in an affine coordinate chart.  Let
$V$ be the vertex of $\Delta$ labeled by $S_V=\{1,2,\ldots,n\}$ and
let $w_1,\ldots,w_n$ be the associated holomorphic coordinates.  Let
$\Omega_\Sigma$ be the symplectic form of $\Pi_\Sigma$ on the open
leaf $\{w_1\dots w_n\not=0\}$ and set
$a=\omega_\Delta^n/\Omega_\Sigma^n=\Pi_\Sigma^n/\pi_\Delta^n$ so
that $\theta_\mu=\Pi_\Sigma^\#(-d\log a)$ there. Then
\[
a=(n+1)(\hbar/h)^n\frac{|w_1|^2\dots |w_n|^2}{(1+\|w\|^2)^{n+1}}
\]
by (\ref{PiSigmaTopPowerLocalCPn}) and
(\ref{PiDeltaTopPowerLocalCPn}) and therefore
\begin{eqnarray*}
d\log a & = &
\sum_{\ell=1}^n\frac{d|w_\ell|^2}{|w_\ell|^2}-(n+1)\sum_{\ell=1}^n
\frac{d |w_\ell|^2}{(1+\|w\|^2)} \\
& = & \sum_{\ell=1}^n
\left(1-\frac{(n+1)|w_\ell|^2}{(1+\|w\|^2)}\right)\frac{d|w_\ell|^2}{|w_\ell|^2}.
\end{eqnarray*}
By factoring out $-(n+1)$ from each coefficient of
$d|w_\ell|^2/|w_\ell|^2=d\log |w_\ell|^2$, ones obtains a familiar
term.  Indeed
\begin{eqnarray*}
d\log a & = & -(n+1)\sum_{\ell=1}^n
\left(\frac{|w_\ell|^2}{(1+\|w\|^2)}-\frac{1}{n+1}\right)d\log|w_\ell|^2
\\
& = & -(n+1)\sum_{\ell=1}^n\langle \Phi_\Delta,u_\ell\rangle
d\log|w_\ell|^2
\end{eqnarray*}
using (\ref{MomentumRelation}). Now
\begin{eqnarray*}
(i\wbar_p w_q \del_{\wbar_p}\wedge\del_{w_q})^\#(d \log |w_\ell|^2)
& = & i\frac{\wbar_p}{\wbar_\ell}\delta_{p\ell} w_q \del_{w_q}-i\frac{w_q}{w_\ell}\delta_{q\ell} \wbar_p \del_{\wbar_p} \\
\end{eqnarray*}
for each $p,q=1,2,\ldots,n$.  Using the symmetry of $B_V$ and the
fact that $B_V$ and $\Phi_\Delta$ are real, one obtains
\begin{eqnarray*}
\Pi_\Sigma^\#(-d\log a) & = & (n+1)\sum_{p,q=1}^n B_{pq} \langle
\Phi_\Delta,u_p\rangle w_q\del_{w_q}+c.c. \\
& = & \frac{(n+1)}{(h/\hbar)}\sum_{q=1}^n \langle
\Phi_\Delta,(uB_V)_q\rangle R_q.
\end{eqnarray*}
Thus, (\ref{ModularVectorFieldonCPn}) is established.
\end{proof}

\begin{figure}\label{Zeros_of_Modular_Field}
\[
\begin{xy}
\xyimport(31,21){\includegraphics[scale=0.8]{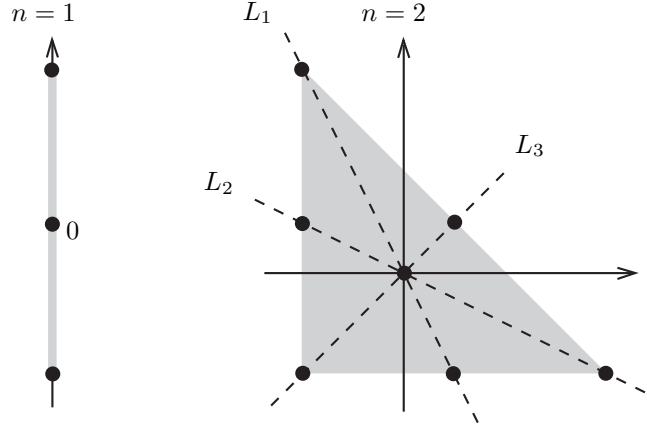}}
,(0,22)*{n=1},(18,22)*{n=2},(1.5,10.5)*{0},(11,22)*{L_1}
,(9,13)*{L_2},(25,15)*{L_3}
\end{xy}\]
\caption{Delzant polytopes for $\CP^n$ for $n=1$ and $n=2$. The
points marked show the momentum map image of the zero set of the
modular vector field $\theta_\mu$ of $\Pi_\Sigma$, computed with
respect to $\mu=\frac{1}{n!}\omega_\Delta^n$, the Delzant Liouville
volume form.  For $n=2$, the annihilator subspaces of $(uB_V)_1$,
$(uB_V)_2$, and $(uB_V)_1-(uB_V)_2$ are indicated by the dashed
lines $L_1$, $L_2$, and $L_3$, respectively.}
\end{figure}

Recall that the centroid of $1$-simplex is its midpoint and the
centroid of an $n$-simplex is defined inductively as the
intersection of the lines connecting the vertices to the centroid of
their opposing facets.  It is a fact that the centroid of a simplex
may be computed as the mean of its vertices.  By construction, the
centroid of $\Delta$ is the origin in $\t^*$.  The following theorem
relates the centroids of the faces of $\Delta$ with the zero locus
of the modular vector field $\theta_\mu$ and is illustrated in
Figure 1.

\begin{theorem}
If $F$ is a face of $\Delta$, then on the $T_\C$-orbit
$\Phi_\Delta^{-1}(F^\circ)$, $\theta_\mu=0$ precisely on the set
$\Phi_\Delta^{-1}(c_F)$ where $c_F$ is the centroid of $F$.
\end{theorem}
\begin{proof}
The claim certainly holds for the fixed points as the vector fields
$R_1,\ldots,R_n$ all vanish there.  For $F=\Delta$ itself, the claim
also holds because $R_1,\ldots,R_n$ are linearly independent on the
open $\T_\C$-orbit $\Phi_\Delta^{-1}(\Delta^\circ)$ so $\langle
\Phi_\Delta,(uB_V)_q\rangle=0$ for each $q=1,2,\ldots,n$.  As $B_V$
is invertible, $uB_V$ is also a frame, so this condition stipulates
that $\Phi_\Delta=0$, i.e., the image under $\Phi_\Delta$ of the
zero locus of $\theta_\mu$ on $\Phi_\Delta^{-1}(\Delta^\circ)$ is
$0\in \t^*$ which is the centroid of $\Delta$ by construction.

Now suppose $F$ is labeled by $\{1,2,\ldots,s\}$.  Then $R_1,\ldots,
R_s$ vanish on $\Phi_\Delta^{-1}(F^\circ)$.   Given the form of
$B_V$, the equations $\langle \eta,u_p\rangle =\frac{1}{n+1}$ for
$p=1,\ldots,s$ and $\langle \eta,(uB_V)_q\rangle=0$ for
$q=s+1,s+2,\ldots,n$ are independent. It suffices then to show that
$\eta=c_F$ is a solution as this solution will be unique.  Let
$V_\ell$ denote the vertex opposite to the facet of the simplex
$\Delta$ labeled by $\ell$. Then the vertices of $F$ are
$\{V_{s+1},V_{s+2},\ldots,V_{n+1}\}$. Now
$V_k=\eta_k-\frac{1}{n+1}\sum_{r=1}^n\eta_r$ for each
$k=1,2,\ldots,n$ while $V_{n+1}=-\frac{1}{n+1}\sum_{r=1}^n \eta_r$.
Therefore, the centroid of $F$ is
\begin{eqnarray*}
c_F & = & \frac{1}{n-s+1}\left(V_{s+1}+V_{s+2}+\ldots
+V_{n+1}\right)
\\
& = &
\frac{1}{n-s+1}\left(\sum_{t=1}^{n-s}\eta_{s+t}-\frac{n-s+1}{n+1}\sum_{r=1}^n\eta_r\right).
\end{eqnarray*}
By construction $\langle c_F,u_p\rangle=\frac{1}{n+1}$ for each
$p=1,2,\ldots,s$. If $q\in\{s+1,s+2,\ldots,n\}$ then
\[
\langle c_F,(uB_V)_q\rangle  =  \frac{1}{n-s+1}\sum_{p=1}^n
B_{pq}\left\langle
\sum_{t=1}^{n-s}\eta_{s+t}-\frac{n-s+1}{n+1}\sum_{r=1}^n\eta_r,
u_p\right\rangle \] \[  =  \frac{1}{n-s+1}\left(\sum_{r=1}^{s}
B_{rq}\left(-\frac{n-s+1}{n+1}\right)+\sum_{t=1}^{n-s}B_{(s+t),q}\left(1-\frac{n-s+1}{n+1}\right)\right).
\]
Given that $B_{pq}=1$ if $q\not=p$ and $B_{pp}=2$, it follows that
\begin{eqnarray*}
\langle c_F,(uB_V)_q\rangle & = &
\frac{1}{n-s+1}\left(s\left(-\frac{n-s+1}{n+1}\right)+(n-s+1)\left(1-\frac{n-s+1}{n+1}\right)\right)
\\
& = & 0
\end{eqnarray*}
for each $q\in\{s+1,s+2,\ldots,n\}$.

Therefore, the image under $\Phi_\Delta$ of the zero locus of
$\theta_\mu$ on $\Phi_\Delta^{-1}(F^\circ)$ is precisely the
centroid $c_F$ of $F$.  This completes the proof.
\end{proof}

\begin{corollary}
The zero locus of $\theta_\mu$ on $\CP^n$ is a disjoint union of
compact tori, one of dimension $\ell$ for each face $F$ of dimension
$\ell$ of $\Delta$.
\end{corollary}
\begin{proof}
By the symplectic convexity theorem, the pre-image under
$\Phi_\Delta$ of a point interior to a face $F$ of $\Delta$ of
dimension $\ell$ is a real torus of dimension $\ell$.  Thus, the
claim follows immediately from the previous theorem.
\end{proof}

What is intriguing about the preceding computations is that the
Poisson structures $\pi_\Delta$ and $\Pi_\Sigma$ are not compatible
when $n>1$ (for $n=1$ the Schouten bracket on $\mathcal V^2$ is
trivial) and yet the modular vector field for $\Pi_\Sigma$ relative
to the Delzant Liouville volume $\frac{1}{n!}\omega_\Delta^{-n}$ can
be computed in terms of moment data.  Although the modular class has
been studied and used as a tool in applications of Poisson geometry,
good examples are lacking, and this formula contributes a natural
family.


\end{document}